\documentclass[11pt, oneside]{article}

\usepackage{geometry}
\usepackage{amsmath,graphicx,amssymb,amsthm,amstext}
\usepackage{caption}
\usepackage{subcaption}
\usepackage{tikz}          		

\newtheorem{theorem}{Theorem}[section]
\newtheorem{lemma}[theorem]{Lemma}

\newtheorem{conjecture}[theorem]{Conjecture}

\newenvironment{definition}[1][Definition]{\begin{trivlist}
\item[\hskip \labelsep {\bfseries #1}]}{\end{trivlist}}

\title{Extremal Numbers for $2 \rightarrow 1$ Directed Hypergraphs with Two Edges Part I: The Nondegenerate Cases}
\author{Alex Cameron}

\begin{document}

\maketitle

\begin{abstract}
Let a $2 \rightarrow 1$ directed hypergraph be a 3-uniform hypergraph where every edge has two tail vertices and one head vertex. For any such directed hypergraph, $F$, let the $n$th extremal number of $F$ be the maximum number of edges that any directed hypergraph on $n$ vertices can have without containing a copy of $F$. There are actually two versions of this problem: the standard version where every triple of vertices is allowed to have up to all three possible directed edges and the oriented version where each triple can have at most one directed edge. In this paper, we determine the standard extremal numbers and the oriented extremal numbers for three different directed hypergraphs. Each has exactly two edges, and of the seven (nontrivial) ($2 \rightarrow 1$)-graphs with exactly two edges, these are the only three with extremal numbers that are cubic in $n$. The standard and oriented extremal numbers for the other four directed hypergraphs with two edges are determined in a companion paper \cite{cameron2015deg}.
\end{abstract}

\section{Introduction}

The combinatorial structure treated in this paper is a $2 \rightarrow 1$ directed hypergraph defined as follows.

\begin{definition}
A \emph{$2 \rightarrow 1$ directed hypergraph} is a pair $H = (V,E)$ where $V$ is a finite set of \emph{vertices} and the set of \emph{edges} $E$ is some subset of the set of all pointed $3$-subsets of $V$. That is, each edge is three distinct elements of $V$ with one marked as special. This special vertex can be thought of as the \emph{head} vertex of the edge while the other two make up the \emph{tail set} of the edge. If $H$ is such that every $3$-subset of V contains at most one edge of $E$, then we call $H$ \emph{oriented}. For a given $H$ we will typically write its vertex and edge sets as $V(H)$ and $E(H)$. We will write an edge as $ab \rightarrow c$ when the underlying $3$-set is $\{a,b,c\}$ and the head vertex is $c$.
\end{definition}

For simplicity from this point on we will always refer to $2 \rightarrow 1$ directed hypergraphs as just \emph{graphs} or sometimes as \emph{$(2 \rightarrow 1)$-graphs} when needed to avoid confusion. This structure comes up as a particular instance of the model used to represent definite Horn formulas in the study of propositional logic and knowledge representation \cite{angluin1992, russell2002}. Some combinatorial properties of this model have been recently studied by Langlois, Mubayi, Sloan, and Gy. Tur\'{a}n in \cite{langlois2009} and \cite{langlois2010}. In particular, they looked at the extremal numbers for a couple of different small graphs. Before we can discuss their results we will need the following definitions.

\begin{definition}
Given two graphs $H$ and $G$, we call a function $\phi:V(H) \rightarrow V(G)$ a homomorphism if it preserves the edges of $H$: \[ab \rightarrow c \in E(H) \implies \phi(a)\phi(b) \rightarrow \phi(c) \in E(G).\] We will write $\phi:H \rightarrow G$ to indicate that $\phi$ is a homomorphism.
\end{definition}

\begin{definition}
Given a family $\mathcal{F}$ of graphs, we say that a graph $G$ is \emph{$\mathcal{F}$-free} if no injective homomorphism $\phi:F \rightarrow G$ exists for any $F \in \mathcal{F}$. If $\mathcal{F} = \{F\}$ we will write that $G$ is $F$-free.
\end{definition}

\begin{definition}
Given a family $\mathcal{F}$ of graphs, let the \emph{$n$th extremal number} $\text{ex}(n,\mathcal{F})$ denote the maximum number of edges that any $\mathcal{F}$-free graph on $n$ vertices can have. Similarly, let the \emph{$n$th oriented extremal number} $\text{ex}_o(n,\mathcal{F})$ be the maximum number of edges that any $\mathcal{F}$-free oriented graph on $n$ vertices can have. Sometimes we will call the extremal number the \emph{standard} extremal number or refer to the problem of determining the extremal number as the \emph{standard version} of the problem to distinguish these concepts from their oriented counterparts. As before, if $\mathcal{F} = \{F\}$, then we will write $\text{ex}(n,F)$ or $\text{ex}_o(n,F)$ for simplicity.
\end{definition}

These are often called Tur\'{a}n-type extremal problems after Paul Tur\'{a}n due to his important early results and conjectures concerning forbidden complete $r$-graphs \cite{turan1941, turan1954, turan1961}. Tur\'{a}n problems for uniform hypergraphs make up a large and well-known area of research in combinatorics, and the questions are often surprisingly difficult.

Extremal problems like this have also been considered for directed graphs and multigraphs (with bounded multiplicity) in \cite{brown1973} and \cite{brown1969} and for the more general directed multi-hypergraphs in \cite{brown1984}. In \cite{brown1969}, Brown and Harary determined the extremal numbers for several types of specific directed graphs. In \cite{brown1973}, Brown, Erd\H{o}s, and Simonovits determined the general structure of extremal sequences for every forbidden family of digraphs analogous to the Tur\'{a}n graphs for simple graphs.

The model of directed hypergraphs studied in \cite{brown1984} have $r$-uniform edges such that the vertices of each edge is given a linear ordering. However, there are many other ways that one could conceivably define a uniform directed hypergraph. The graph theoretic properties of a more general definition of a nonuniform directed hypergraph were studied by Gallo, Longo, Pallottino, and Nguyen in \cite{gallo1993}. There a directed hyperedge was defined to be some subset of vertices with a partition into head vertices and tail vertices.

Recently in \cite{cameron2016}, this author tried to capture many of these possible definitions for ``directed hypergraph" into one umbrella class of relational structures called generalized directed hypergraphs. The structures in this class include the uniform and simple versions of undirected hypergraphs, the totally directed hypergraphs studied in \cite{brown1984}, the directed hypergraphs studied in \cite{gallo1993}, and the $2 \rightarrow 1$ model studied here and in \cite{langlois2009,langlois2010}.

In \cite{langlois2009, langlois2010}, they study the extremal numbers for two different graphs with two edges each. They refer to these two graphs as the 4-resolvent and the 3-resolvent configurations after their relevance in propositional logic. Here, we will denote these graphs as $R_4$ and $R_3$ respectively and define them formally as \[V(R_4) = \{a,b,c,d,e\} \text{ and } E(R_4) = \{ab \rightarrow c, cd \rightarrow e\}\] and \[V(R_3) = \{a,b,c,d\} \text{ and } E(R_3) = \{ab \rightarrow c, bc \rightarrow d\}.\]

In \cite{langlois2010} the authors determined $\text{ex}(n,R_4)$ for sufficiently large $n$, and in \cite{langlois2009} they determined a sequence of numbers asymptotically equivalent to the sequence of numbers $\text{ex}(n,R_3)$ as $n$ increases to infinity. In these papers, the authors discuss a third graph with two edges which they call an Escher configuration because it calls to mind the Escher piece where two hands draw each other. This graph is on four vertices, $\{a,b,c,d\}$ and has edge set $\{ab \rightarrow c,cd \rightarrow b\}$. We will denote it by $E$. These three graphs actually turn out to be the only three nondegenerate graphs with exactly two edges on more than three vertices.

\begin{definition}
A graph $H$ is \emph{degenerate} if its vertices can be partitioned into three sets, $V(H) = T_1 \cup T_2 \cup K$ such that every edge of $E(H)$ is of the form $t_1t_2 \rightarrow k$ for some $t_1 \in T_1$, $t_2 \in T_2$, and $k \in K$.
\end{definition}

An immediate consequence of a result shown in \cite{cameron2016} is that the extremal numbers for a graph $H$ are cubic in $n$ if and only if $H$ is not degenerate.

In our model of directed hypergraph, there are nine different graphs with exactly two edges. Of these, four are not degenerate. One of these is the graph on three vertices with exactly two edges, $V=\{a,b,c\}$ and $E=\{ab \rightarrow c, ac \rightarrow b\}$. It is trivial to see that both the standard and oriented extremal numbers for this graph are ${n \choose 3}$. The other three nondegenerate graphs are $R_4$, $R_3$, and $E$. We will determine both the standard and oriented extremal numbers for each of these graphs in Sections 2, 3, and 4 respectively. The extremal numbers for the five degenerate cases are determined in a companion paper \cite{cameron2015deg}.

The proofs that follow rely heavily on the concept of a link graph. For undirected $r$-graphs, the link graph of a vertex is the $(r-1)$-graph induced on the remaining vertices such that each $(r-1)$-set is an $(r-1)$-edge if and only if that set together with the specified vertex makes an $r$-edge in the original $r$-graph \cite{keevash2011}. In the directed hypergraph model here, there are a few ways we could define the link graph of a vertex. We will need the following three.

\begin{definition}
Let $x \in V(H)$ for some graph $H$. The \emph{tail link graph} of $x$ $T_x$ is the simple undirected 2-graph on the other $n-1$ vertices of $V(H)$ with edge set defined by all pairs of vertices that exist as tails pointing to $x$ in some edge of $H$. That is, $V(T_x) = V(H) \setminus \{x\}$ and \[ E(T_x) = \{yz : yz \rightarrow x \in H\}.\] The size of this set, $|T_x|$ will be called the \emph{tail degree} of $x$. The degree of a particular vertex $y$ in the tail link graph of $x$ will be denoted $d_x(y)$.

Similarly, let $D_x$ be the \emph{directed link graph} of $x$ on the remaining $n-1$ vertices of $V(H)$. That is, let $V(D_x) = V(H) \setminus \{x\}$ and \[E(D_x) = \{y \rightarrow z : xy \rightarrow z \in E(H)\}.\]

And let $L_x$ denote the \emph{total link graph} of $x$ on the remaining $n-1$ vertices: $V(L_x) = V(H) \setminus \{x\}$ and \[E(L_x) = E(T_x) \cup E(D_x).\] So $L_x$ is a partially directed 2-graph.
\end{definition}

\section{The 4-resolvent graph $R_4$}

\begin{figure}
  \centering
      \begin{tikzpicture}
			\filldraw [black] (-2,2) circle (1pt);
			\filldraw [black] (-2,0) circle (1pt);
			\draw[thick] (-2,2) -- (-2,0);
			\filldraw [black] (0,1) circle (1pt);
			\draw[thick, ->] (-2,1) -- (0,1);
			\filldraw [black] (0,-1) circle (1pt);
			\draw[thick] (0,1) -- (0,-1);
			\draw[thick,->] (0,0) -- (2,0);
			\filldraw [black] (2,0) circle (1pt);
		\end{tikzpicture}
  \caption{$R_4$}
  \label{C}
\end{figure}

In \cite{langlois2009}, the authors gave a simple construction for an $R_4$-free graph. Partition the vertices into sets $T$ and $K$ and take all possible edges with tail sets in $T$ and head vertex in $K$. When there are $n$ vertices, this construction gives ${t \choose 2}(n-t)$ edges where $t=|T|$. This is optimized when $t = \left\lceil \frac{2n}{3} \right\rceil$. In \cite{langlois2010}, they showed that this number of edges is maximum for $R_4$-free graphs for sufficiently large $n$ and that the construction is the unique extremal $R_4$-free graph.

\begin{figure}
	\centering
	\begin{tikzpicture} [scale=0.5]
		\filldraw[color=black,fill=blue!5] (0,0) ellipse (2 and 3);
		\node at (-2,3) {$T$};
		
		\filldraw[color=black,fill=blue!5] (6,0) ellipse (1 and 2);
		\node at (7,2) {$K$};
		
		\node [left] at (-2,0) {${n-k \choose 2}$ tail pairs};
		\node [right] at (7,0) {$k$ heads};

		\filldraw[black] (0,-1) circle (1pt);
		\filldraw[black] (0,1) circle (1pt);
		\filldraw[black] (6,0.5) circle (1pt);
		\draw[thick] (0,-1) -- (0,1);
		\draw[thick,->] (0,0) -- (6,0.5);
	\end{tikzpicture}
	\caption{The lower bound construction for a graph with no $R_4$.}
\end{figure}
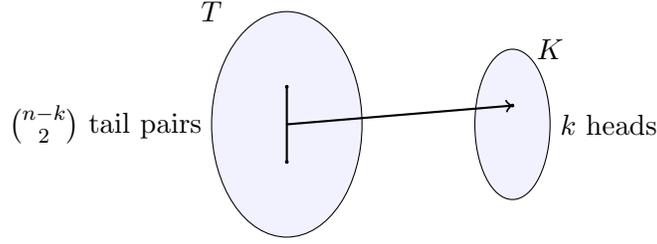

We now give an alternate shorter proof that  $\left\lfloor \frac{n}{3} \right\rfloor{\left\lceil \frac{2n}{3} \right\rceil \choose 2}$ is an upper bound on the extremal number for $R_4$ for sufficiently large $n$ in both the standard and oriented versions of the problem. The proof also establishes the uniqueness of the construction.

\begin{theorem}
\label{TypeC}
For all $n \geq 29$, \[ex_o(n,R_4) = \left\lfloor \frac{n}{3} \right\rfloor{\left\lceil \frac{2n}{3} \right\rceil \choose 2}\]and for all $n \geq 70$, \[ex(n,R_4) = \left\lfloor \frac{n}{3} \right\rfloor{\left\lceil \frac{2n}{3} \right\rceil \choose 2}.\]Moreover, in each case there is one unique extremal construction up to isomorphism when $n \equiv 0,1 \text{ mod } 3$ and exactly two when $n \equiv 2 \text{ mod } 3$.
\end{theorem}

\begin{proof}
In either the standard or the oriented model, let $H$ be an $R_4$-free graph on $n$ vertices. Partition $V(H)$ into sets $T \cup K \cup B$ where $T$ is the set of vertices that appear in tail sets of edges but never appear as the head of any edge, $K$ is the set of vertices that do not belong to any tail set, and $B$ is the set that appear as both heads and tails.

If $B$ is empty, then $H$ is a subgraph of the lower bound construction and we are done. So assume that there exists some $v \in B$. The link graph $L_v$ must contain at least one undirected edge and at least one directed edge. If any undirected edge is independent from any directed edge in $L_v$, then $v$ would be the intersection vertex for an $R_4$ in $H$. Therefore, every directed edge in $L_v$ is incident to every undirected edge.

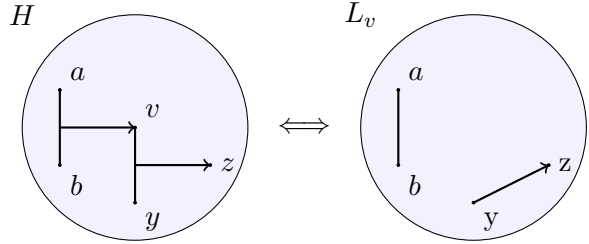
\begin{figure}
	\centering
		\begin{tikzpicture} [scale=0.5]
			\filldraw[color=black, fill=blue!5] (0,1) circle [radius=3];
			\node at (-3,4) {$H$};
		
			\filldraw [black] (-2,2) circle (1pt);
			\filldraw [black] (-2,0) circle (1pt);
			\draw[thick] (-2,2) -- (-2,0);
			\filldraw [black] (0,1) circle (1pt);
			\draw[thick, ->] (-2,1) -- (0,1);
			\filldraw [black] (0,-1) circle (1pt);
			\draw[thick] (0,1) -- (0,-1);
			\draw[thick,->] (0,0) -- (2,0);
			\filldraw [black] (2,0) circle (1pt);
			\node [above right] at (0,1) {$v$};
			\node [right] at (2,0) {$z$};
			\node [below right] at (0,-1) {$y$};
			\node [above right] at (-2,2) {$a$};
			\node [below right] at (-2,0) {$b$};
			
			\node at (4.5,1) {$\iff$};
			
			\filldraw[color=black, fill=blue!5] (9,1) circle [radius=3];
			\node at (6,4) {$L_v$};
			
			\filldraw [black] (7,2) circle (1pt);
			\node [above right] at (7,2) {$a$};
			\filldraw [black] (7,0) circle (1pt);
			\node [below right] at (7,0) {$b$};
			\draw[thick] (7,2) -- (7,0);
			\filldraw[black] (9,-1) circle (1pt);
			\node [below right] at (9,-1) {y};
			\filldraw[black] (11,0) circle (1pt);
			\node [right] at (11,0) {z};
			\draw[thick,->] (9,-1) -- (11,0);
		\end{tikzpicture}
	\caption{$H$ has an $R_4$ if and only if the link graph of some vertex $v$ contains a directed edges and an undirected edge that do not intersect.}
\end{figure}

We want to show that if $v \in B$, then $|E(L_v)| = O(n)$. Determining an upper bound on the number of edges in $L_v$ is equivalent to determining an upper bound on the number of red and blue edges on $n-1$ vertices such that each red edge is incident to each blue edge and there is at least one edge of each color.

If we are working in the oriented model where multiple edges on the same triple are not allowed then no pair of vertices in $L_x$ can hold more than one edge. If we are working in the standard model, then two vertices in this graph may have up to three edges between them, say two red and one blue.

First, let's consider the oriented version. In this case we have at least one edge of each color and they must be incident. So let $xy$ be blue and let $yz$ be red. Then all other edges must be incident to $x$, $y$, or $z$. Moreover, any edge from $x$ to the remaining $n-4$ vertices must be red since it is independent from $yz$ and any edge from $z$ to the remaining $n-4$ must be blue. Therefore, there are at most $2(n-4)$ edges from $\{x,y,z\}$ to the remaining $n-4$ vertices.

\begin{figure}
  \centering
  	\begin{tikzpicture}
		\filldraw[black] (0,6) circle (1pt);
		\node [above] at (0,6) {$y$};
		
		\filldraw[black] (-2,5) circle (1pt);
		\node [left] at (-2,5) {$x$};
		
		\filldraw[black] (2,5) circle (1pt);
		\node [right] at (2,5) {$z$};
		
		\draw[blue, thick] (0,6)--(-2,5);
		\draw[red, thick] (0,6)--(2,5);
		
		\filldraw[color=black, fill=blue!5] (0,2) ellipse (4 and 2);
		\node at (0,1) {$n-4$};
		
		\filldraw[black] (-3,2) circle (1pt);
		\filldraw[black] (-2,2) circle (1pt);
		\filldraw[black] (-1,2) circle (1pt);
		\node at (0.5,2) {$\cdots$};
		\filldraw[black] (2,2) circle (1pt);
		\filldraw[black] (3,2) circle (1pt);
		
		\draw[red, thick,dashed] (0,6)--(-1,2);
		\draw[blue, thick,dashed] (0,6)--(2,2);
		
		\draw[red, thick,dashed] (-2,5)--(-1,2);
		\draw[red, thick,dashed] (-2,5)--(-3,2);
		
		\draw[blue, thick,dashed] (2,5)--(3,2);
		\draw[blue, thick,dashed] (2,5)--(-2,2);
		
	\end{tikzpicture}
  \caption{A simple graph on $n-1$ vertices with red and blue edges such that each red edge is incident to each blue edge and there is at least one blue edge, $xy$, and at least one red edge, $yz$, can have no edge contained in the remaining $n-4$ vertices. Moreover, only red edges can go from $x$ to the remaining vertices and only blue edges can go from $z$ to the remaining vertices.}
\end{figure}
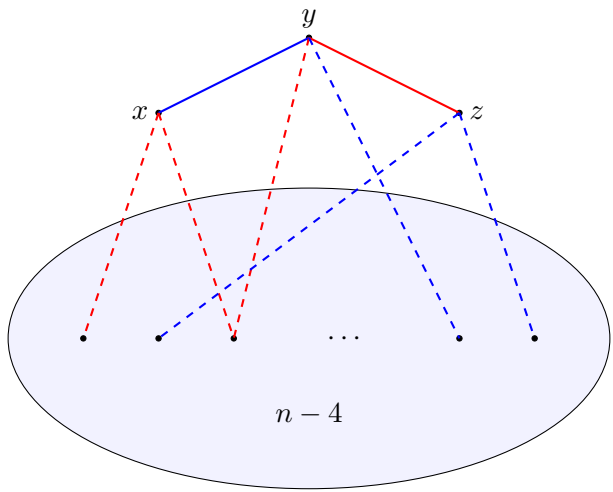

In the standard case our initial two red and blue edges may either be incident as before with $xy$ blue and $yz$ red or they might be incident in two vertices so that $xy$ holds both a red and a blue edge. If none of the first type of incidence exists, then there are at most 3 edges, all on $xy$.

\begin{figure}
	\centering
	\begin{tikzpicture}
		\filldraw[black] (0,6) circle (1pt);
		\node [above] at (0,6) {$y$};
		
		\filldraw[black] (-1,4) circle (1pt);
		\node [left] at (-1,4) {$x$};
		
		\filldraw[black] (1,4) circle (1pt);
		\node [right] at (1,4) {$z$};
		
		\draw[blue, thick] (0,6)--(-1,4);
		\draw[red, thick] (0,6)--(1,4);
		
		\filldraw[black] (5,6) circle (1pt);
		\node [right] at (5,6) {$y$};
		
		\filldraw[black] (5,4) circle (1pt);
		\node [right] at (5,4) {$x$};
		
		\draw[blue, thick] (5,6) .. controls (4.5,5) .. (5,4);
		\draw[red, thick] (5,6) .. controls (5.5,5) .. (5,4);
	\end{tikzpicture}
	\caption{When two vertices are allowed to have up to two red edges and one blue edge, then an adjacent red and blue edge pair is either incident in one or two vertices.}
\end{figure}
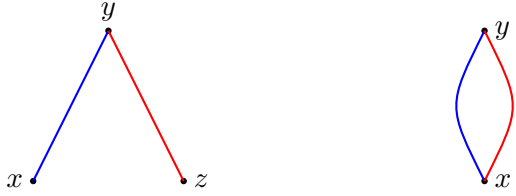

So let's assume that the first type of incidence exists - $xy$ is a blue edge and $yz$ is a red edge. As before, all other edges must be incident to these three vertices, any edge from $x$ to the remaining $n-4$ vertices must be red, and any edge from $z$ to these vertices must be blue. Since each pair can have up to two reds and one blue, then there are at most $2(n-4)$ edges from $x$, $3(n-4)$ edges from $y$, and $n-4$ from $z$ to the remaining vertices. Therefore, there are at most $5(n-4)$ additional edges.

In either the standard or oriented versions of the problem, edges that do not contain vertices of $B$ must have their tails in $T$ and their heads in $K$. So there are at most \[\left\lfloor \frac{n-b}{3} \right\rfloor{\left\lceil \frac{2(n-b)}{3} \right\rceil \choose 2}\] edges that do not intersect $B$. Hence, \[|E(H)| < \left\lfloor \frac{n-b}{3} \right\rfloor{\left\lceil \frac{2(n-b)}{3} \right\rceil \choose 2} + cnb\]where $c=2$ in the oriented case and $c=5$ in the standard case.

This expression is maximum on $b \in [0,n]$ only at the endpoint $b=0$ for all $n \geq 29$ when $c=2$ and for all $n \geq 70$ when $c=5$.

Therefore, we can never do better than the lower bound construction. Moreover, since $B$ must be empty to reach this bound, then the construction is unique when $n \equiv 0,1 \text{ mod } 3$. When $n \equiv 2 \text{ mod } 3$, then \[\left\lfloor \frac{n}{3} \right\rfloor{\left\lceil \frac{2n}{3} \right\rceil \choose 2}=\left\lceil \frac{n}{3} \right\rceil{\left\lfloor \frac{2n}{3} \right\rfloor \choose 2}\]so there are exactly two extremal constructions in that case.
\end{proof}

\section{The 3-resolvent graph $R_3$}

\begin{figure}
	\centering
	\begin{tikzpicture}
	\filldraw[black] (-1,0) circle (1pt);
	\filldraw[black] (1,0) circle (1pt);
	\filldraw[black] (0,2) circle (1pt);
	\filldraw[black] (2,1.75) circle (1pt);
	
	\draw[thick] (-1,0)--(1,0);
	\draw[thick,->] (0,0)--(0,2);
	\draw[thick] (1,0)--(0,2);
	\draw[thick,->] (0.5,1)--(2,1.75);
	\end{tikzpicture}
	\caption{$R_3$}
\end{figure}

In \cite{langlois2009}, the authors gave a simple construction for an $R_3$-free graph. Partition the vertices into sets $A$ and $B$ and take all possible edges with a tail set in $A$ and head vertex in $B$ plus all possible edges with a tail set in $B$ and a head in $A$. When there are $n$ vertices, this construction gives $(n-a){a \choose 2} + a{n-a \choose 2}$ edges where $a=|A|$. This is optimized when $a = \left\lceil \frac{n}{2} \right\rceil$. The authors showed that this number of edges is asymptotically optimal for $R_3$-free graphs.

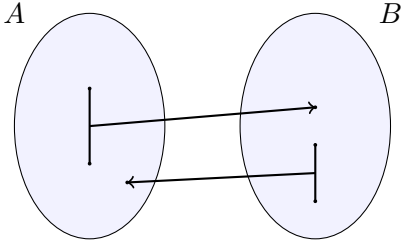
\begin{figure}
	\centering
	\begin{tikzpicture} [scale=0.5]
		\filldraw[color=black,fill=blue!5] (0,0) ellipse (2 and 3);
		\node at (-2,3) {$A$};
		
		\filldraw[color=black,fill=blue!5] (6,0) ellipse (2 and 3);
		\node at (8,3) {$B$};

		\filldraw[black] (0,-1) circle (1pt);
		\filldraw[black] (0,1) circle (1pt);
		\filldraw[black] (6,0.5) circle (1pt);
		\draw[thick] (0,-1) -- (0,1);
		\draw[thick,->] (0,0) -- (6,0.5);
		\filldraw[black] (6,-0.5) circle (1pt);
		\filldraw[black] (6,-2) circle (1pt);
		\filldraw[black] (1,-1.5) circle (1pt);
		\draw[thick] (6,-0.5) -- (6,-2);
		\draw[thick,->] (6,-1.25) -- (1,-1.5);
	\end{tikzpicture}
	\caption{The unique $R_3$-free extremal construction.}
\end{figure}

We show that in both the standard and the oriented versions of this problem that this construction is in fact the best that we can do. We will start with the oriented case since it is less technical.

 \subsection{The oriented version}
 
\begin{theorem}
\label{exE}
For all $n$, \[ex_o(n,R_3) = \left\lfloor \frac{n}{2} \right\rfloor \left\lceil \frac{n}{2} \right\rceil \frac{n-2}{2}.\]Moreover, there is one unique extremal $R_3$-free construction up to isomorphism for each $n$.
\end{theorem}

\begin{proof}
Let $H$ be an $R_3$-free graph on $n$ vertices. Consider the total link graph, $L_x$, for some $x \in V(H)$. If \[yz, z \rightarrow t \in E(L_x)\] or if \[y \rightarrow z, z \rightarrow t \in E(L_x)\] then $H$ is not $R_3$-free (See Figure~\ref{forLx}).

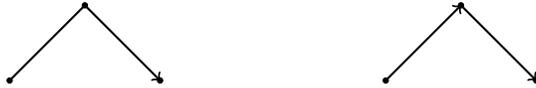
\begin{figure}
	\centering
	\begin{tikzpicture}
	\filldraw[black] (0,0) circle (1pt);
	\filldraw[black] (1,1) circle (1pt);
	\filldraw[black] (2,0) circle (1pt);
	\draw[thick] (0,0) -- (1,1);
	\draw[thick,->] (1,1) -- (2,0);
	
	\filldraw[black] (5,0) circle (1pt);
	\filldraw[black] (6,1) circle (1pt);
	\filldraw[black] (7,0) circle (1pt);
	\draw[thick,->] (5,0) -- (6,1);
	\draw[thick,->] (6,1) -- (7,0);
	\end{tikzpicture}
	\caption{Forbidden intersection types in $L_x$ for any vertex $x$ in an $R_3$-free graph.}
	\label{forLx}
\end{figure}

Let $U_x \subseteq V(L_x)$ be the set of vertices that appear as the tail vertex of some directed edge in $L_x$. Then no edges of $L_x$ can be contained entirely inside $U_x$ - it is an independent set with respect to both directed and undirected edges. Moreover, all undirected edges of $L_x$ must appear entirely within the complement, $C_x := V(L_x) \setminus U_x$. Hence, if we let $u_x = |U_x|$, then \[2|E(H)| = \sum_{x \in V(H)} |D_x| \leq \sum_{x \in V(H)} u_x(n-1-u_x).\]

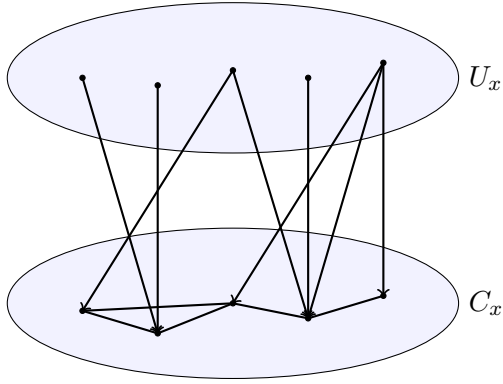
\begin{figure}
	\centering
	\begin{tikzpicture}
		\filldraw[color=black,fill=blue!5] (0,0) ellipse (3 and 1);
		\filldraw[color=black,fill=blue!5] (0,3) ellipse (3 and 1);
		
		\node [right] at (3,0) {$C_x$};
		\node [right] at (3,3) {$U_x$};
		
		\filldraw[black] (0,3.1) circle (1pt);
		\filldraw[black] (-1,2.9) circle (1pt);
		\filldraw[black] (1,3) circle (1pt);
		\filldraw[black] (2,3.2) circle (1pt);
		\filldraw[black] (-2,3) circle (1pt);
		
		\filldraw[black] (0,0) circle (1pt);
		\filldraw[black] (-1,-0.4) circle (1pt);
		\filldraw[black] (1,-0.2) circle (1pt);
		\filldraw[black] (2,0.1) circle (1pt);
		\filldraw[black] (-2,-0.1) circle (1pt);
		
		\draw[thick,->] (0,3.1)--(1,-0.2);
		\draw[thick,->] (0,3.1)--(-2,-0.1);
		\draw[thick,->] (1,3)--(1,-0.2);
		\draw[thick,->] (2,3.2)--(2,0.1);
		\draw[thick,->] (2,3.2)--(1,-0.2);
		\draw[thick,->] (2,3.2)--(0,0);
		\draw[thick,->] (-1,2.9)--(-1,-0.4);
		\draw[thick,->] (-2,3)--(-1,-0.4);
		
		\draw[thick] (-2,-0.1)--(-1,-0.4);
		\draw[thick] (-1,-0.4)--(0,0);
		\draw[thick] (0,0)--(1,-0.2);
		\draw[thick] (1,-0.2)--(2,0.1);
		\draw[thick] (-2,-0.1)--(0,0);
	\end{tikzpicture}
	\caption{The structure of $L_x$ for any $x$ in an $R_3$-free graph}
	\label{Lxstructure}
\end{figure}

Each term of this sum is maximized when $u_x \in \left\{\left\lfloor \frac{n-1}{2} \right\rfloor, \left\lceil \frac{n-1}{2} \right\rceil\right\}$. Therefore, the result is immediate if $n$ is even. The situation is slightly more complicated for odd $n$.

In this case, \[u_x(n-1-u_x) \leq \left(\frac{n-1}{2}\right)^2\] for each $x$. However, we need $u_x = \frac{n-1}{2}$ in order to attain this maximum value. This means that there are $\frac{n-1}{2}$ vertices in $C_x$, and so there are at most ${\frac{n-1}{2} \choose 2}$ edges in $T_x$. Therefore, if every $x \in V(H)$ gave $u_x = \frac{n-1}{2}$, then \[|E(H)| = \sum_{x \in V(H)} |T_x| < \frac{(n-2)(n-1)(n+1)}{8},\] which is given by the construction we want to show is optimal.

For each $x$ let $i_x \in \{0,\ldots,\frac{n-1}{2}\}$ be the integer such that \[u_x(n-1-u_x) = \left(\frac{n-1}{2} - i_x\right)\left(\frac{n-1}{2}+i_x\right).\]Then, \[|E(H)| \leq \frac{1}{2} \sum_{x \in V(H)} \left(\frac{n-1}{2} - i_x\right)\left(\frac{n-1}{2}+i_x\right)= \frac{n(n-1)^2}{8} - \frac{1}{2} \sum_{j = 0}^{\frac{n-1}{2}} k_j j^2\] where $k_j$ is the number of vertices $x \in V(H)$ for which $i_x=j$.

Since the construction gives $\frac{(n-2)(n-1)(n+1)}{8}$ for odd $n$, then we are only interested in beating this. So set \[\frac{(n-2)(n-1)(n+1)}{8} \leq \frac{n(n-1)^2}{8} - \frac{1}{2} \sum_{j = 0}^{\frac{n-1}{2}} k_j j^2.\] This gives \[\sum_{j = 0}^{\frac{n-1}{2}} k_j j^2 \leq \frac{n-1}{2}.\]

Since we can also find $|E(H)|$ by counting the number of undirected edges over the $L_x$, then we can upper bound the number of these by assuming $u_x = \frac{n-1}{2} - i_x$ for each $x$ since this increases the size of $C_x$. This gives \[|E(H)| \leq \sum_{x \in V(H)} {\frac{n-1}{2} + i_x \choose 2}= \frac{n^3 - 4n^2 + 3n}{8} + \frac{1}{2} \sum_{j = 0}^{\frac{n-1}{2}} j(n+j-2)k_j.\]

We can also set this greater than or equal to the known lower bound: \[\frac{(n-2)(n-1)(n+1)}{8} \leq \frac{n^3 - 4n^2 + 3n}{8} + \frac{1}{2} \sum_{j = 0}^{\frac{n-1}{2}} j(n+j-2)k_j\] to get \[\frac{(n-1)^2}{2} \leq \sum_{j = 0}^{\frac{n-1}{2}} k_j j^2 + (n-2) \sum_{j = 0}^{\frac{n-1}{2}} k_j j.\] Combining the inequalities gives \[0 \leq \sum_{j = 0}^{\frac{n-1}{2}} k_j (j - j^2).\] Since $j - j^2 < 0$ for any $j \geq 2$ and $j - j^2 = 0$ when $j=0,1$, then $k_j = 0$ for all $j \geq 2$.

Moreover, once all these are set to zero we get that \[k_1 \leq \frac{n-1}{2} \leq k_1.\] Therefore, $k_1 = \frac{n-1}{2}$ and so $k_0 = \frac{n+1}{2}$ since $\sum k_j = n$. This gives the desired upper bound.

Now we can show that the lower bound construction is the unique extremal example up to isomorphism. Let $H$ be an extremal example on $n$ vertices, and define a relation, $\sim$, on the vertices such that $x \sim y$ if and only if either $x = y$ or $y \in U_x$. This defines an equivalence relation on $V(H)$. Reflexivity and symmetry are both immediate. For transitivity note that the proof of the upper bound requires that every possible directed edge be taken from $U_x$ to $C_x$ for each $x \in V(H)$. Therefore, if we assume towards a contradiction that $y \in U_x$ and $z \in U_y$ but $z \not \in U_x$, then $z \in C_x$. So $xy \rightarrow z \in E(H)$ which means $z \in C_y$, a contradiction.

When $n$ is even there must be exactly two equivalence classes each of size $\frac{n}{2}$. Similarly, when $n$ is odd there must be two equivalence classes of sizes $\frac{n-1}{2}$ and $\frac{n+1}{2}$. Therefore, the lower bound construction must be unique.
\end{proof}

\subsection{The standard version}

\begin{theorem}
\label{TypeE}
For all $n \geq 6$, \[ex(n,R_3) = \left\lfloor \frac{n}{2} \right\rfloor \left\lceil \frac{n}{2} \right\rceil \frac{n-2}{2}.\]Moreover, there is one unique extremal $R_3$-free construction up to isomorphism for each $n$.
\end{theorem}

\begin{proof}
Let $H$ be an $R_3$-free graph on $n$ vertices. Let $x \in V(H)$, and call any pair of vertices in $L_x$ a multiedge if they contain more than one edge. Let $V(L_x) = U_x \cup C_x \cup M_x$ where $M_x$ is the set of vertices that are incident to multiedges (that is, the minimal subset of vertices that contains all multiedges) and $U_x$ and $C_x$ are defined on the rest of the vertices as in Theorem~\ref{exE}. The goal is to show that if $M_x$ is nonempty for any vertex $x$, then $H$ has strictly fewer than the number of edges in the unique oriented construction given in Theorem~\ref{exE}. Therefore, that construction must be the unique extremal $R_3$ example for the standard problem as well.

There are three possibilities for multiedges in $M_x$: two oppositely directed edges, one directed edge and one undirected edge, and one undirected edge with two oppositely directed edges. If $y,z \in M_x$ have two directed edges between them, then neither $y$ nor $z$ is incident to any other edge in $L_x$ since any incidence would create one of the two forbidden edge incidences of $L_x$ as discussed in the previous theorem.

If $y$ and $z$ have only one directed edge (assume it is $y \rightarrow z$) and one undirected edge between them, then $y$ cannot be incident to any more edges for the same reason as before, but $z$ can be incident to undirected edges as well as directed edges with $z$ at the head. This means that $z$ may be the vertex of intersection of a star of these types of multiedges within $M_x$, and between any two such stars, the vertices of intersection may have an undirected edge between them, but no directed.

\begin{figure}
	\centering
	\begin{tikzpicture}
		\filldraw[color=black,fill=blue!5] (1.5,-0.5) circle [radius=3];
		\node at (-1.5,2.5) {$M_x$};
		
		\filldraw[black] (0,0) circle (1pt);
		\filldraw[black] (1,0) circle (1pt);
		\filldraw[black] (0,-1) circle (1pt);
		\filldraw[black] (-0.707,-0.707) circle (1pt);
		\filldraw[black] (0,1) circle (1pt);
		\draw[thick, ->] (1,0) -- (0,0);
		\draw[thick, ->] (0,-1) -- (0,0);
		\draw[thick, ->] (-0.707,-0.707) -- (0,0);
		\draw[thick, ->] (0,1) -- (0,0);
		
		\filldraw[black] (2,2) circle (1pt);
		\filldraw[black] (1,2) circle (1pt);
		\draw[thick, ->] (1,2) -- (2,2);
		
		\filldraw[black] (2,1) circle (1pt);
		\filldraw[black] (2,0) circle (1pt);
		\filldraw[black] (3,0) circle (1pt);
		\draw[thick, ->] (3,0) -- (2,0);
		\draw[thick, ->] (2,1) -- (2,0);
		
		\filldraw[black] (0,-1.7) circle (1pt);
		\filldraw[black] (0,-2.7) circle (1pt);
		\draw[thick,<->] (0,-1.7) -- (0,-2.7);
		\filldraw[black] (1,-1.7) circle (1pt);
		\filldraw[black] (1,-2.7) circle (1pt);
		\draw[thick,<->] (1,-1.7) -- (1,-2.7);
		\filldraw[black] (2,-1.7) circle (1pt);
		\filldraw[black] (2,-2.7) circle (1pt);
		\draw[thick,<->] (3,-1.7) -- (3,-2.7);
		\filldraw[black] (3,-1.7) circle (1pt);
		\filldraw[black] (3,-2.7) circle (1pt);
		\draw[thick,<->] (2,-1.7) -- (2,-2.7);
	\end{tikzpicture}
	\caption{Example structure of $M_x$ with 3 single directed edge stars and 4 double directed pairs}
	\label{Mx}
\end{figure}
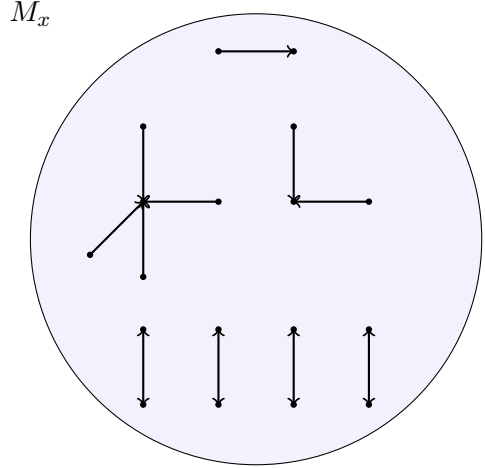

Therefore, the structure of the internal directed edges of $M_x$ looks like Figure~\ref{Mx} with only the vertices of intersection of the single directed edge stars able to accept more edges from the rest of $L_x$. Directed edges from the rest of the graph to $M_x$ must originate in $U_x$. Therefore, if $M_x$ consists of $d$ double directed edge pairs of vertices and $k$ single directed stars with the $i$th star containing $s_i$ vertices, then the total number of directed edges incident to vertices of $M_x$ is at most \[2d + \sum_{i=1}^k (s_i - 1 + u)\]where $u$ is the number of vertices in $U_x$.

If we assume that $M_x$ is nonempty, then $|M_x| = m \geq 2$. The number of directed edges incident to or inside of $M_x$ is at most $m+k(u-1)$. Therefore, for $u \geq 2$, the number of directed edges incident to vertices of $M_x$ is maximized when the number of single directed edge stars is maximized. This is $\left\lfloor \frac{m}{2} \right\rfloor$ stars. Therefore, there are at most \[ \frac{m}{2} (u+1) \]directed edges incident to vertices of $M_x$. Thus, if $|C_x|=c$, then $L_x$ can have at most $uc + \frac{m}{2} (u+1)$ directed edges. And since $u \geq 2$, then \[uc + \frac{m}{2}(u+1) < u(c+m).\]

So $L_x$ has strictly less directed edges than a complete bipartite graph on the same number of vertices would. In Theorem~\ref{exE} every $L_x$ needed to be a complete bipartite graph in terms of the directed edges in order for the maximum number of edges to be obtained, and only in the case of odd $n$ could some of these bipartitions be less than equal or almost equal. In those cases the parts could only have $\frac{n-1}{2} - 1$ and $\frac{n-1}{2} + 1$ vertices. Therefore, the only way that $u(c+m)$ could have more than this is if $u = c+m$ and so $u = \frac{n-1}{2}$.

We assume that $m \geq 2$ and $u \geq 2$, but if both are equal to 2, then $c = u-m =0$ and $n=4$, a contradiction since $n$ is odd. Therefore, one of them must be strictly greater. So \[uc + \frac{m}{2}(u+1) < (u-1)(u+1) = \left(\frac{n-1}{2} - 1\right)\left(\frac{n-1}{2} + 1\right).\] This leaves only the cases where $u=0$ and $u=1$ which are trivial.

So every link graph of $H$ that contains a multiedge has strictly less than $(\frac{n-1}{2})^2-1$ directed edges. This is enough to prove that an extremal $R_3$-free graph on an even number of vertices must be oriented. However, if there are an odd number of vertices it is possible that there could be enough directed link graphs with the maximum $\left(\frac{n-1}{2}\right)^2$ directed edges to make up the deficit for the directed link graphs with strictly less than $\left(\frac{(n-3)(n+1)}{4}\right)$ due to multiedges.

In this case there would need to be at least $\frac{n+3}{2}$ vertices with directed link graphs that are complete bipartite graphs with parts of size $\frac{n-1}{2}$ each. Let $S$ be the set of these vertices. For any $x,y \in S$ define the relation $x \sim y$ if and only if $y \in U_x$. As in the proof of Theorem~\ref{exE}, this turns out to be an equivalence relation. By the definition of $S$ one equivalence class can hold at most $\frac{n+1}{2}$ vertices. So there must be two nonempty classes. Let these classes be $A$ and $B$ with $a$ and $b$ vertices respectively. Let $C$ be the set of vertices that are in every $U_x$ for $x \in A$ but not in $A$ itself. Similarly, let $D$ be the set of vertices that are in every $U_x$ for $x \in B$ but are not in $B$ itself. The sets $A$, $B$, $C$, and $D$ are disjoint. Let $c=|C|$ and $d=|D|$. Then \[a+c = b+d = \frac{n+1}{2},\] but $a+b+c+d \leq n$, a contradiction. This is enough to show the result.
\end{proof}

\section{The Escher graph $E$}

\begin{figure}
	\centering
	\begin{tikzpicture}
		\filldraw[black] (0,0) circle (1pt);
		\filldraw[black] (0,2) circle (1pt);
		\filldraw[black] (2,1) circle (1pt);
		\filldraw[black] (2,-1) circle (1pt);
		
		\draw[thick] (0,0) -- (0,2);
		\draw[thick] (2,-1) -- (2,1);
		\draw[thick,->] (0,1) -- (2,1);
		\draw[thick,->] (2,0) -- (0,0);
	\end{tikzpicture}
	\caption{$E$}
	\label{Gpic}
\end{figure}

In this section, we will prove the following result on the maximum number of edges of a $E$-free.

\begin{theorem}
\label{mexG}
For all $n$, \[ex(n,E) = {n \choose 3} + 2\]and there are exactly two extremal construction up to isomorphism for each $n \geq 4$.
\end{theorem}

But first we will prove the easier oriented version of the problem. This result will be needed to prove Theorem~\ref{mexG}.

\subsection{The oriented version}

\begin{theorem}
\label{exG}
For all $n$, \[ex_o(n,E) = {n \choose 3}\]and there is exactly one extremal construction up to isomorphism.
\end{theorem}

\begin{proof}
The upper bound here is trivial so we need only come up with an $E$-free construction that uses ${n \choose 3}$ edges. Let $H$ be the directed hypergraph defined on vertex set $V(H) = [n]$ and edge set, \[E(H) = \left \{ ab \rightarrow c : a <b<c \right \}.\] That is take some linear ordering on the $n$ vertices and for each triple direct the edge to the largest vertex. Then every triple has an edge and $H$ contains no copy of $E$.

Now we will show that this construction is unique. Let $H$ be an $E$-free graph on $n$ vertices and ${n \choose 3}$ edges. Define a relation on the vertices, $\prec$, where $x \prec y$ if and only if there exists an edge in $E(H)$ with $x$ in the tail and $y$ as the head vertex. Then $\prec$ is a partial ordering of the vertices that is almost linear in that every pair of vertices are comparable except for the two smallest elements (see Figure ~\ref{almostlin}).
\end{proof}

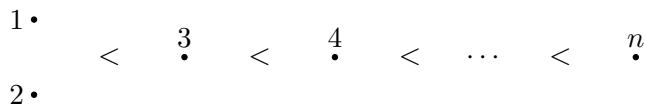
\begin{figure}
	\centering
	\begin{tikzpicture}
		\filldraw[black] (0,0.5) circle (1pt);
		\filldraw[black] (0,1.5) circle (1pt);
		\filldraw[black] (2,1) circle (1pt);
		\filldraw[black] (4,1) circle (1pt);
		\filldraw[black] (8,1) circle (1pt);
		
		\node at (1,1) {$<$};
		\node at (3,1) {$<$};
		\node at (5,1) {$<$};
		\node at (6,1) {$\cdots$};
		\node at (7,1) {$<$};
		
		\node [left] at (0,1.5) {$1$};
		\node [left] at (0,0.5) {$2$};
		\node [above] at (2,1) {$3$};
		\node [above] at (4,1) {$4$};
		\node [above] at (8,1) {$n$};
	\end{tikzpicture}
	\caption{An ``almost" linear ordering on the vertices of an $E$-free directed hypergraph.}
	\label{almostlin}
\end{figure}

We now shift our attention to the the standard version of the problem where a triple of vertices can have more than one edge. Here, both of the lower bound constructions are similar to the unique extremal construction in the oriented version.

\subsection{Two lower bound constructions for $\text{ex}(n,E)$}

The first construction is the same as the extremal construction in the oriented case but with two additional edges placed on the ``smallest" triple. That is, let $H_1 = ([n],E_1)$ where \[E_1 =  \left \{ ab \rightarrow c : a <b<c \right \} \cup \{13 \rightarrow 2, 23 \rightarrow 1\}.\] See Figure~\ref{mexG1}.

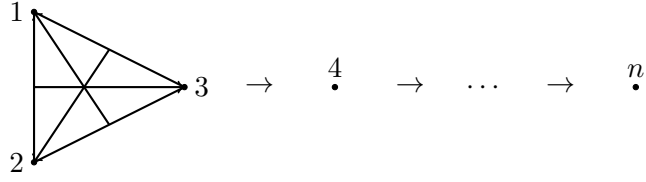
\begin{figure}
	\centering
	\begin{tikzpicture}
		\filldraw[black] (0,0) circle (1pt);
		\filldraw[black] (0,2) circle (1pt);
		\filldraw[black] (2,1) circle (1pt);
		\filldraw[black] (4,1) circle (1pt);
		\filldraw[black] (8,1) circle (1pt);
		
		\draw[thick] (0,0) -- (0,2);
		\draw[thick] (0,0) -- (2,1);
		\draw[thick] (0,2) -- (2,1);
		\draw[thick,->] (0,1) -- (2,1);
		\draw[thick,->] (1,0.5) -- (0,2);
		\draw[thick,->] (1,1.5) -- (0,0);
		
		\node at (3,1) {$\rightarrow$};
		\node at (5,1) {$\rightarrow$};
		\node at (6,1) {$\cdots$};
		\node at (7,1) {$\rightarrow$};
		
		\node [left] at (0,2) {$1$};
		\node [left] at (0,0) {$2$};
		\node [right] at (2,1) {$3$};
		\node [above] at (4,1) {$4$};
		\node [above] at (8,1) {$n$};
	\end{tikzpicture}
	\caption{The first extremal construction, $H_1$, for an $E$-free directed hypergraph on $n$ vertices.}
	\label{mexG1}
\end{figure}

Moreover, it is important to note that if an $E$-free graph with ${n \choose 3} + 2$ edges has at least one edge on every vertex triple, then it must be isomorphic to $H_1$. This is because we can remove two edges to get an $E$-free subgraph where each triple has exactly one edge. Therefore, this  must be the unique extremal construction established in Theorem~\ref{exG}. The only way to add two edges to this construction and avoid creating an Escher graph is to add the additional edges to the smallest triple under the ordering.

The second construction is also based on the oriented extremal construction. Let $H_2=([n],E_2)$ where \[E_2 =  \left( E_1 \setminus \{23 \rightarrow 4, 23 \rightarrow 1\} \right) \cup \{14 \rightarrow 2, 14 \rightarrow 3\}.\] See Figure~\ref{mexG2}.

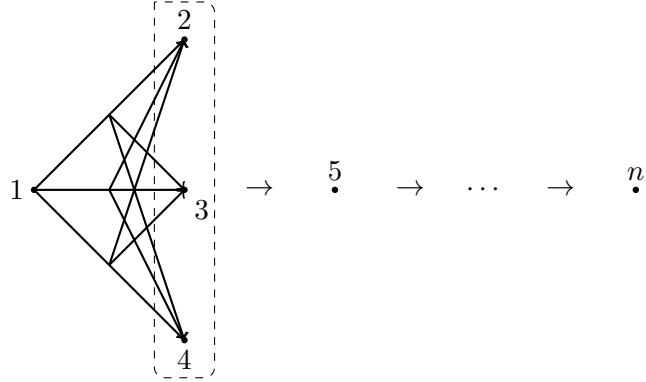
\begin{figure}
	\centering
	\begin{tikzpicture}
		\filldraw[black] (0,0) circle (1pt);
		
		\filldraw[black] (2,0) circle (1pt);
		\filldraw[black] (2,2) circle (1pt);
		\filldraw[black] (2,-2) circle (1pt);
		
		\filldraw[black] (4,0) circle (1pt);
		\filldraw[black] (8,0) circle (1pt);
		
		\draw[thick] (0,0) -- (2,0);
		\draw[thick] (0,0) -- (2,2);
		\draw[thick] (0,0) -- (2,-2);
		\draw[thick,->] (1,1) -- (2,0);
		\draw[thick,->] (1,1) -- (2,-2);
		\draw[thick,->] (1,0) -- (2,2);
		\draw[thick,->] (1,0) -- (2,-2);
		\draw[thick,->] (1,-1) -- (2,2);
		\draw[thick,->] (1,-1) -- (2,0);
		
		\node at (3,0) {$\rightarrow$};
		\node at (5,0) {$\rightarrow$};
		\node at (6,0) {$\cdots$};
		\node at (7,0) {$\rightarrow$};
		
		\node [left] at (0,0) {$1$};
		
		\node [above] at (2,2) {$2$};
		\node [below right] at (2,0) {$3$};
		\node [below] at (2,-2) {$4$};
		
		\node [above] at (4,0) {$5$};
		\node [above] at (8,0) {$n$};
		
		\draw[rounded corners, dashed] (1.6,2.5) -- (1.6,-2.5) -- (2.4,-2.5) -- (2.4,2.5) -- (1.6,2.5);
	\end{tikzpicture}
	\caption{The second extremal construction, $H_2$, for an $E$-free graph on $n$ vertices.}
	\label{mexG2}
\end{figure}

For the rest of this section we will show that any $E$-free graph is either isomorphic to one of these two constructions or has fewer than ${n \choose 3} + 2$ edges. Roughly speaking, the strategy is to take any $E$-free graph and show that we can add and remove edges to it so that we preserve $E$-freeness, remove most multiple edges from triples that had more than one, and never decrease the overall number of edges.

\subsection{Add and Remove Edges}

Let $H$ be an $E$-free graph and represent its vertices as the disjoint union of three sets: \[V(H) = D \cup R \cup T\] where $D$ (for `Done') is the set of all vertices that have complete graphs on three or more vertices as tail link graphs, $R$ (for `Ready to change') is the set of vertices not in $D$ that have at least three edges in their tail link graphs, and $T$ is the set of all other vertices (those with `Two or fewer edges in their tail link graphs').

The plan is now to remove and add edges in order make a new graph $H'$ which is also $E$-free, has at least as many edges as $H$, and whose vertices make a disjoint union, \[V(H') = D' \cup T'\] where $D'$ and $T'$ are defined exactly the same as $D$ and $T$ except in terms of the vertices of $H'$.

That is, for each vertex $x \in R$, we will add all possible edges to complete $T_x$. This moves $x$ from $R$ to $D$. The edges removed will be all those that pointed from $x$ to a vertex that points to $x$. This will destroy triples with more than one edge as we go. The following observation will ensure that this procedure only ever moves vertices from $R$ to $D$, from $R$ to $T$, from $R$ to $R$, and from $T$ to $T$. Since each step moves one vertex from $R$ to $D$ and ends when $R$ is empty, then the procedure is finite. Here is the observation:

\begin{lemma}
\label{observation}
Let $H$ be an $E$-free graph, and let $x,y \in V(H)$. If $d_x(y),d_y(x) > 0$, then $d_x(y) = d_y(x) = 1$. In other words, for any two vertices, $x$ and $y$, if $d_y(x) \geq 2$, then $d_x(y) = 0$.
\end{lemma}

\begin{proof}
Suppose not. Let $d_x(y),d_y(x) > 0$ and suppose $d_x(y) \geq 2$. Then there exist two distinct vertices, $a$ and $b$ such that \[ay \rightarrow x, by \rightarrow x \in E(H).\] There also exists a vertex $c$ such that $xc \rightarrow y \in E(H)$. Since $c$ must be distinct from either $a$ or $b$ if not both, then this yields an Escher graph.
\end{proof}

Now, let us make the procedure slightly more formal: While there exist vertices in $R$, pick one, $x \in R$, and for each pair $a,b \in V(T_x)$, add the edge $ab \rightarrow x$ to $E(H)$ if it is not already an edge. Then, for each $a \in V(T_x)$, remove all edges of $E(H)$ of the form $xs \rightarrow a$ for any third vertex $s$.

Since there were at least three edges in $T_x$, then the added edges will move $x$ from $R$ to $D$. The removed edges, if any, will only affect vertices in $R$ or in $T$ since if $xs$ is removed from $T_a$, then this implies that $a \in T_x$ and that $x \in T_a$ and so both had degree one in the other's tail link graph. Hence, $a \not \in D$. Moreover, an affected vertex in $R$ will either stay in $R$ or move to $T$ while an affected vertex in $T$ will stay in $T$ since it is only losing edges from its tail link graph.

Moreover, at the end of this process $D'$ will contain no triple of vertices with more than one edge. Therefore, the only such triples of vertices of $H'$ will be entirely in $T'$ or will consist of vertices from both $T'$ and $D'$. We will show later that there cannot be too many of these triples. First, we need to show that after each step of this procedure, no Escher graph is created and at least as many edges are added to the graph as removed.

\subsection{No copy of $E$ is created and the number of edges can only increase}

Fix a particular vertex $x \in R$ to move to $D$. Add and remove all of the designated edges. Suppose that we have created an Escher graph. Since the only edges added point to $x$, then the configuration must be of the form, $ab \rightarrow x, xc \rightarrow a$ for some distinct vertices, $a$, $b$, and $c$. Therefore, $a \in V(T_x)$ and so $xc \rightarrow a$ would have been removed in the process.

Now we will show that at least as many edges have been added to $H$ as removed by induction on the number of independent edges in $T_x$. Start by assuming there are 0 independent edges in $T_x$ and assume that there are $k$ vertices in $T_x$ that have degree one. Then at most $k$ edges will be removed. If $k=0$, then no edges are removed and there is a strict increase in the number of edges.

If $k=1$, then let $y_1$ be the degree one vertex and let $y_2$ be the vertex it is incident to. Since $d_x(y_2) \neq 1$ and $d_x(y_2) \geq 1$, then $d_x(y_2) \geq 2$. So there exists a third vertex, $y_3$, and similarly, $d_x(y_3) \geq 2$ but $y_2$ is not adjacent to $y_1$. Hence, there exists a fourth vertex, $y_4$. So at most one edge is removed and at least two edges are added, $y_1y_3 \rightarrow x$ and $y_1y_4 \rightarrow x$. Therefore, there is a strict increase in the number of edges.

If $k = 2$, then the fact that $T_x$ has at least three edges means that there must be at least two additional vertices in $T_x$. Hence, at most two edges are removed but at least three are added. If $k \geq 3$, then at most $k$ are removed but ${k \choose 2}$ are added which nets \[{k \choose 2} - k = \frac{k(k-3)}{2} \geq 0\] edges added.

Now, for the induction step, assume that $T_x$ has $m>0$ independent edges and that the process on a $T_x$ with $m-1$ independent edges adds just as many edges as it removes. Let $yz$ be an independent edge in $T_x$ and let $A$ be the set of vertices of $T_x$ that are not $y$ or $z$. Since $T_x$ has at least three edges, then $A$ contains at least three vertices. Therefore, the number of added edges is at least 6 between $A$ and $\{y,z\}$. The number of edges removed from $T_y$ and $T_z$ together is at most 2. By assumption, the number of edges removed from the other tail link graphs of vertices in $A$ is offset by the number of edges added inside $A$. Therefore, there is a strict increase in the number of edges.

To summarize, we have shown that $H'$ is an $E$-free graph such that \[|E(H)| \leq |E(H')|\] and \[V(H') = D' \cup T'\] such that any triple of vertices of $H'$ with more than one edge must intersect the set $T'$. We will now consider what is happening in $T'$ by cases.

\subsection{Case 1: $|T'| \geq 5$}

Let $T' = \{x_1,x_2,\ldots,x_t\}$ for $t \geq 5$. For each $x_i$ remove all edges of $H'$ that have $x_i$ as a head. By the definition of $T'$ this will remove at most $2t$ edges from $H'$.

Next, add all edges to $T'$ that follow the index ordering. That is, for each triple $\{x_i,x_j,x_k\}$ add the edge that points to the largest index, $x_ix_j \rightarrow x_k$ where $i<j<k$. This will add ${t \choose 3}$ edges. The new graph has \[{t \choose 3} - 2t \geq 0\] more edges than $H'$. Moreover, it is $E$-free and oriented. Therefore, $|E(H)| < {n \choose 3}$.

\subsection{Case 2: $|T'| \leq 4$ and there exists an $x \in T'$ such that $T_x$ is two independent edges}

Assume that some $x \in T'$ has a tail link graph $T_x$ such that $ab, cd \in E(T_x)$ for four distinct vertices, $\{a,b,c,d\}$. If \[d_a(x) = d_b(x) = d_c(x) = d_d(x) = 1,\] then $a,b,c,d,x \in T'$, a contradiction of the assumption that $|T'| \leq 4$.

Therefore, we can add the edges \[ac \rightarrow x, ad \rightarrow x, bc \rightarrow x, bd \rightarrow x\] and remove any edges that point to a vertex from $\{a,b,c,d\}$ with $x$ in the tail set. Because $x$ has zero degree in at least one of those tail link graphs, then we have removed at most three edges and added four, a strict increase. We have also not created any triples of vertices with more than one edge or any Escher graphs.

We may now assume that $|T'| \leq 4$ and that the tail link graphs of vertices in $T'$ are never two independent edges.

\subsection{Case 3: $|T'| =0,1,2$}

First, note that if $H'$ has a triple with more than one edge $\{x,y,z\}$ then at least two of its vertices must be in $T'$ as a consequence of Lemma~\ref{observation}. Therefore, if $|T'|=0,1$, then $H'$ is oriented and so \[|E(H)| \leq |E(H')| < {n \choose 3}.\]

Moreover, if $T'=\{x,y\}$ and $H'$ is not oriented, then any vertex triple with more than one edge must have two edges of the form, \[zx \rightarrow y,zy \rightarrow x\] for some third vertex $z$. If there exist two such vertices $z_1 \neq z_2$ that satisfy this, then there would be an Escher graph. Hence, there is at most one vertex triple with more than one edge and it would have at most two edges. Therefore, \[|E(H)| \leq |E(H')| \leq {n \choose 3}+1.\]

\subsection{Case 4: $|T'| =3$}

First, suppose that there exists a triple, $\{x,y,z\}$ with all three possible edges. Then $T' = \{x,y,z\}$. Since any triple with multiple edges must intersect $T'$ in at least two vertices, then any additional such triple would make an Escher graph with one of the edges in $T'$. Therefore, $H'$ has exactly one triple of vertices with all three edges on it and no others. So \[|E(H) \leq |E(H')| \leq {n \choose 3} + 2.\] Moreover, to attain this number of edges, no triple of vertices can be empty of edges. In this case,$H'$ must be isomorphic to the first construction $H_1$.

Next, assume that no triple of vertices has all three edges and let $T' = \{x,y,z\}$. Therefore, $H'$ needs at least two triples of vertices that each hold two edges or else \[|E(H)| \leq |E(H')| \leq {n \choose 3} + 1\] automatically. Suppose one of the multiedges is $\{x,y,z\}$ itself. Then without loss of generality let the edges be $xy \rightarrow z$ and $xz \rightarrow y$. The second triple with two edges must have its third vertex in $D'$. Call this vertex $v$.The vertex $x$ cannot be in this second triple of vertices without creating an Escher graph. So the edges must be $vy \rightarrow z$ and $vz \rightarrow y$. But this also creates an Escher graph.

Therefore, neither of the two triples that hold two edges are contained entirely within $T'$. So without loss of generality they must be $vx \rightarrow y, vy \rightarrow x$ and $wy \rightarrow z, wz \rightarrow y$. If $v \neq w$, then $vx, wz \in T_y$, a contradiction to our assumption that $T'$ contains no vertices with tail link graphs that are two independent edges. Hence, $v=w$.

Since $v \in D$, then $T_v$ has at least three vertices. Moreover, since $v$ is in the tail link graphs of each vertex of $T'$, then none of these vertices can be in $T_v$. Remove all edges pointing to the vertices of $T'$. This is at most 6 edges. Add all possible edges with $v$ as the head and a tail set among the set $V(T_v) \cup \{x,y,z\}$. This adds at least 12 new edges. The new graph is oriented and $E$-free. Therefore, $|E(H)| <  {n \choose 3}$.

\subsection{Case 5: $|T'| =4$}

First, assume that there is some triple $\{x,y,z\}$ that contains all three possible edges. As before, there are no additional triples with more than one edge. So \[|E(H)| \leq |E(H')| \leq {n \choose 3} + 2.\] The first construction $H_1$ is the unique extremal construction under this condition since all triples must be used at least once.

So assume that all triples with more than one edge have two edges each. Then we must have at least two. Assume that one of them is contained within $T' = \{a,b,c,d\}$. Without loss of generality let it be $ab \rightarrow c, ac \rightarrow b$. Since the second such triple intersects $T'$ in at least two vertices, then it must intersect $\{a,b,c\}$ in at least one vertex.

If it intersects $\{a,b,c\}$ in two vertices, then without loss of generality (to avoid a copy of $E$) the second triple must be of the form $ab \rightarrow x, ax \rightarrow b$. Hence, $x \in T'$ so $x=d$.

But now there is no edge possible on $\{b,c,d\}$. Therefore, there must be a third such triple for $H'$ to have ${n \choose 3} +2$ edges. This triple must be $ac \rightarrow d, ad \rightarrow c$. And the only way to actually make it to the maximum number of edges now must be to have an edge on every other triple.

Every triple of the form $\{b,c,s\}$ for $s \in D$ must have the edge $bc \rightarrow s$ since the other two options would create an Escher graph. Similarly, $bd \rightarrow s$ and $cd \rightarrow s$ are the only options for triples of the form $\{b,d,s\}$ and $\{c,d,s\}$ respectively. Next, any triple of the form $\{a,b,s\}$ must hold the edge $ab \rightarrow s$ since the other two edges create Escher graphs. Similarly, every triple of the forms $\{a,c,s\}$ and $\{a,d,s\}$ must hold the edges $ac \rightarrow s$ and $ad \rightarrow s$ respectively.

Since each triple contained in $D$ holds exactly one edge, then the induced subgraph on $D$ must be isomorphic to the oriented extremal example of an $E$-free graph on $n-4$ vertices. Therefore, the entire graph $H'$ must be isomorphic to the second extremal construction $H_2$ in order to attain ${n \choose 3} + 2$ edges.

So assume that the second triple with two edges intersects $\{a,b,c\}$ in only one vertex. Then these edges must be $xa \rightarrow d, xd \rightarrow a$. This can be the only additional triple with two edges. So to make it to ${n \choose 3} +2$ edges we need each triple to have an edge. However, the edge for $\{a,b,d\}$ is forced to be $ad \rightarrow b$ and the edge for $\{b,c,d\}$ is forced to be $bc \rightarrow d$. This makes an Escher graph. So \[|E(H)| \leq |E(H')| \leq {n \choose 3} + 1.\]

Now assume that no vertex triple with multiple edges is contained entirely within $T'$, but assume that there are at least two such triples in $H'$. The only way that two triples could have distinct vertices in $D'$ is if they were of the forms (without loss of generality), $xa \rightarrow b, xb \rightarrow a$, and $yc \rightarrow d, yd \rightarrow c$. Otherwise, the pairs of the two triples that are in $T'$ would intersect resulting in either a copy of $E$ (if both triples use the same pair) or a vertex in $T'$ with two independent edges as a tail link graph.

So there must be exactly two such triples. Therefore, all other triples of vertices must contain exactly one edge in order to reach ${n \choose 3}+2$ edges overall. To avoid the forbidden subgraph this edge must be $ab \rightarrow c$ for the triple $\{a,b,c\}$ and $cd \rightarrow a$ for the triple $\{a,c,d\}$. But this is an Escher graph. Hence, not all triples may be used and so \[|E(H)| \leq |E(H')| \leq {n \choose 3} + 1.\]

Therefore, we may now assume for each multiedge triple that the vertex from $D'$ is always $x$. First, assume that there are only two such triples. As before, if we assume that the only two such triples are $xa \rightarrow b, xb \rightarrow a$ and $xc \rightarrow d, xd \rightarrow c$, then there can be not be an edge on both $\{a,b,c\}$ and $\{a,c,d\}$. Hence, there would be a suboptimal number of edges overall.

On the other hand, if the only two such triples are adjacent in $T'$, then they are, without loss of generality, $xa \rightarrow b, xb \rightarrow a$ and $xb \rightarrow c, xc \rightarrow a$. In this case, no edge can go on the triple $\{a,b,c\}$ at all and so there are at most ${n \choose 3} + 1$ edges overall.

Therefore, we must assume there are at least three such triples that meet at $x$. If these three triples make a triangle in $T'$, then they are $xa \rightarrow b, xb \rightarrow a$, $xb \rightarrow c, xc \rightarrow b$, and $xc \rightarrow a, xa \rightarrow c$. Again, there can be no edges on the triple $\{a,b,c\}$. Hence, every other triple must hold an edge to attain ${n \choose 3} + 2$ edges overall.

On the triple $\{a,b,d\}$ this edge must be $ab \rightarrow d$ to avoid making a copy of $E$. Similarly, we must have the edges $ac \rightarrow d$ and $bc \rightarrow d$. But this means that $d \not \in T'$, a contradiction.

On the other hand, if there are three triples of vertices with more than one edge on each that do not make a triangle in $T'$ or if there are four or more such triples, then $x$ is in the tail link graphs for each vertex in $T'$. Hence, none of these vertices may be in the tail link graph, $T_x$. However, $x \in D'$ so its tail link graph has at least three vertices. Remove all edges pointing to vertices of $T'$ (at most 8). Add all edges pointing to $x$ with tail sets in $T'$ (6 new edges) and between $T'$ and $V(T_x)$ (at least 12 new edges). So this adds at least ten edges to $H'$ to create $H''$. $H''$ is oriented so \[|E(H)| < |E(H'')| \leq {n \choose 3}.\]

This exhausts all of the cases and establishes that \[\text{ex}(n,E) = {n \choose 3} +2\] with exactly two extremal examples up to isomorphism.

\section{Conclusion}

In \cite{brown1969}, Brown and Harary started studying extremal problems for directed $2$-graphs by determining the extremal numbers for many ``small" digraphs and for some specific types of digraphs such as tournaments - a digraph where every pair of vertices has exactly one directed edge. We could follow their plan of attack in studying this $2 \rightarrow 1$ model and also look for the extremal numbers of tournaments. Here, a tournament would be a graph with exactly one directed edge on every three vertices. In particular, a transitive tournament might be an interesting place to begin. A transitive tournament is a tournament where the direction of each edge is based on an underlying linear ordering of the vertices as in the oriented lower bound construction in Theorem~\ref{exG}.

Denote the $2 \rightarrow 1$ transitive tournament on $k$ vertices by $TT_k$. Since the ``winning" vertex of the tournament will have a complete $K_{k-1}$ as its tail link graph, then any $H$ on $n$ vertices for which each $T_{x}$ is $K_{k-1}$-free must be $TT_k$-free. Therefore, \[n \left(\frac{n-1}{k-2}\right)^2 {k-2 \choose 2} \leq \text{ex}(n,TT_k), \text{ex}_o(n,TT_k).\] This also immediately shows that the transitive tournament on four vertices with the ``bottom" edge removed has this extremal number exactly.

\begin{theorem}
Let $TT_4^-$ denote the graph with vertex set $V(TT_4^-) = \{a,b,c,d\}$ and edge set \[E(TT_4^-)=\{ab \rightarrow d, bc \rightarrow d, ac \rightarrow d\}.\] Then \[\text{ex}(n,TT_4^-) = n \left\lfloor \frac{n-1}{2} \right\rfloor \left\lceil \frac{n-1}{2} \right\rceil.\]
\end{theorem}

Is it still true if we add an edge to $\{a,b,c\}$?

\begin{conjecture}
Let $TT_4$ denote the graph with vertex set $V(TT_4) = \{a,b,c,d\}$ and edge set \[E(TT_4)=\{ab \rightarrow d, bc \rightarrow d, ac \rightarrow d, ab \rightarrow c\}.\] Then \[\text{ex}(n,TT_4) = n \left\lfloor \frac{n-1}{2} \right\rfloor \left\lceil \frac{n-1}{2} \right\rceil.\]
\end{conjecture}

Another way of generalizing the extremal questions asked in this paper is to ask about $r \rightarrow 1$ models of directed hypergraphs. If we look at every $(r \rightarrow 1)$-graph with exactly two edges we see that these fall into four main types of graph. Let $i$ be the number of vertices that are in the tail set of both edges. Then let $T_r(i)$ denote the graph where both edges point to the same head vertex, let $H_r(i)$ denote the graph where the edges point to different head vertices neither of which are in the tail set of the other, let $R_r(i)$ denote the graph where the first edge points to a head vertex in the tail set of the second edge and the second edge points to a head not in the tail set of the first edge, and let $E_r(i)$ denote the graph where both edges point to heads in the tail sets of each other. So in terms of the graphs discussed in this paper, the 3-resolvent would be a $R_2(1)$, the 4-resolvent would be a $R_2(0)$, the Escher graph would be a $E_2(0)$, and two edges on the same triple of vertices would be an $E_2(1)$.

The nondegenerate cases here would be $R_r(i)$ and $E_r(i)$. It would be interesting to find the extremal numbers for these graphs in general. To what extent do the current proofs extend to these graphs? For example, any $r \rightarrow 1$ transitive tournament on $n$ vertices would be $E_r(i)$-free. This solves the oriented version and gives a lower bound for the standard version: \[\text{ex}_o(n,E_r(i)) = {n \choose r+1}.\]

For the generalized resolvent configurations, the lower bound extremal constructions for $R_3$ and $R_4$ generalize easily to the $r \rightarrow 1$ setting, but are they ever tight? Can we generalize these constructions to get extremal numbers for all $R_r(i)$?

\bibliography{nondegenerate}
\bibliographystyle{plain}

\end{document}